\documentclass[a4paper,12pt]{amsart}
\usepackage[latin1]{inputenc}
\usepackage[T1]{fontenc}
\usepackage[english]{babel}
\usepackage{mathrsfs}
\usepackage{amscd}
\usepackage{amsfonts}
\usepackage{amsmath}
\usepackage{amssymb}
\usepackage{amstext}
\usepackage{amsthm}
\usepackage{amsbsy}
\usepackage{xspace}
\usepackage[all]{xy}
\usepackage{graphicx}
\usepackage{url}
\usepackage{latexsym}

\newtheorem{prop}{Proposition}[section]
\newtheorem{esempio}[prop]{Examples}
\newtheorem{thm}[prop]{Theorem}
\newtheorem{lemma}[prop]{Lemma}
\newtheorem{corollario}[prop]{Corollary}
\newtheorem{defini}[prop]{Definition}
\newtheorem{osserva}[prop]{Remark}

\newtheorem*{mainresult}{Main Theorem}

\theoremstyle{remark}
\newtheorem*{fact}{Fact}

\newcommand{\ie}{i.e\mbox{.}\xspace}

\newcommand{\rddots}[1]{\cdot^{\cdot^{\cdot^{#1}}}}

\def\be{\begin{equation}}
\def\ee{\end{equation}}

\newcommand*{\set}[1]{\{#1\}}
\newcommand*{\abs}[1]{\lvert#1\rvert}

\newcommand{\C}{\mathbb C}
\newcommand{\Z}{\mathbb Z}
\newcommand{\N}{\mathbb N}
\newcommand{\R}{\mathbb R}
\newcommand{\Q}{\mathbb Q}

\DeclareMathOperator{\td}{t.d.}

\newcommand{\xt}{\tilde x}
\newcommand{\at}{\tilde a}

\newcommand{\x}{\bar x}
\newcommand{\qv}{\bar q}
\newcommand{\kv}{\bar k}

\title[]{Generic solutions of equations with iterated exponentials}
\author{P. D'Aquino}
\address{Dipartimento di Matematica e Fisica, Seconda Universit\`a di Napoli, Viale Lincoln 5, 81100 Caserta, Italy}
\email{paola.daquino@unina2.it}
\author{A. Fornasiero}
\address{Dipartimento di Matematica e Informatica, Universit\`{a} di Parma, Parco Area delle Scienze, 53/A, 43124 Parma, Italy}
\email{antongiulio.fornasiero@gmail.com}
\author{G. Terzo}
\address{Dipartimento di Matematica e Fisica, Seconda Universit\`a di Napoli, Viale Lincoln 5, 81100 Caserta, Italy}
\email{giuseppina.terzo@unina2.it}
\date{\today}

\begin{document}


\maketitle

\begin{abstract}
\noindent We study solutions of exponential polynomials over the complex field. Assuming Schanuel's Conjecture we prove that certain polynomials of the form
\[ p(z, e^z, e^{e^z}, e^{e^{e^{z}}}) = 0, \] 
have generic solutions in $\Bbb C$. 
\end{abstract}

\section{Introduction}

We consider analytic functions over $\C$ of the following form
\begin{equation}
\label{function}
f(z) = p(z, e^z, e^{e^z}, \ldots, e^{e^{e^{\rddots{e^{z}}}}})
\end{equation}
where $p(x, y_1, \ldots, y_k) \in \mathbb C[x, y_1, \ldots, y_k]$, and we investigate the existence of a solution $a$ which is generic, \ie such that
\[
\td_{\Q}(a, e^a, e^{e^a}, \ldots, e^{e^{e^{\rddots{e^{a}}}}}) = k,
\]
 where $k$ is the number of iterations of exponentation which appear in the polynomial $p.$\\

\noindent {\bf Conjecture. }
{\it Let $p(x, y_1, \ldots, y_k)$ be a nonzero irreducible polynomial in
$\C[x, y_1, \ldots, y_k]$, depending on $x$ and the last variable $y_k$. 
Then
$$p(z, e^z, e^{e^z}, \ldots, e^{e^{e^{\ldots^{e^{z}}}}})= 0$$ has a generic solution in $\Bbb C.$}

\medskip

A result of Katzberg (see \cite{katz}) implies that \eqref{function} has always infinitely many zeros unless the polynomial is of a certain form, see Section \ref{casocomplesso}. 
Hence, the main problem is to prove the existence of a solution which is generic.
In this context a fundamental role is often played by a conjecture in transcendental number theory due to Schanuel which concerns the exponential function.

\smallskip

\noindent {\bf Schanuel's Conjecture (SC):} 
Let $\lambda_1, \ldots, \lambda_n \in \C$ be
linearly independent over $\Q$. 
Then $\Q(\lambda_1, \ldots, \lambda_n, e^{\lambda_1},
\ldots, e^{\lambda_n})$ has transcendence degree $(\td_{\Q})$ at least $n$ over
$\Q$.

\medskip
\noindent (SC) includes Lindemann-Weierstrass Theorem. The analogous statement for the ring of power series $t\C[[t]]$ has been proved by Ax in \cite{ax}.\\
Schanuel's Conjecture has played a crucial role in exponential
algebra (see \cite{angusfree}, \cite{GiusyEuler}, \cite{dmt2}), and in the
model theory of exponential fields (see \cite{MacWilkie},
\cite{zilber}, \cite{marker}, \cite{dmt}, \cite{dmt1}).

\medskip

Assuming Schanuel's Conjecture, we are able to prove some particular cases of the
Conjecture.

\begin{mainresult}
(SC) Let $p(x, y_1, y_2, y_3) \in \Q^{alg}[x, y_1, y_2, y_3]$ be a nonzero irreducible polynomial depending on $x$ and the last variable.
Then, there exists a  generic solution of 
\[
p(z, e^z, e^{e^z}, e^{e^{e^{z}}}) = 0.
\]
\end{mainresult}

In fact, we obtain infinitely many generic solutions.
We prove analogous results for polynomials $p(z, e^{e^{z}})$ and $p(z,e^z,e^{e^z})$ (see Theorem \ref{dueiterazioni} and Theorem \ref{lem:generic-exp2}). In the general case for $k>3$ iterations of exponentiation we have only partial results (see Proposition \ref{gradotras2}).

One of the main ingredients in the proof of the above theorem is a result due to Masser on the existence of zeros of systems of exponential equations (see Section \ref{masser'sresult}). Only very recently (in private correspondence with D. Masser) we have become aware that these ideas  have been developed further in a recent preprint \cite{BM}  where the authors  show the existence of solutions of certain  exponential polynomials. Some methodology is different from what we use in this paper, and moreover they are not interested in generic solutions.

\medskip
 
One of our motivations for studying generic solutions of exponential polynomials comes from a fascinating analysis of the complex exponential field
\[
(\C, +, \cdot, 0, 1,e^z),\]
due to Zilber \cite{zilber}. Zilber identified a class of algebraically closed fields of characteristic $0$ equipped with an exponential function. His axioms include Schanuel's Conjecture, and are inspired by the
complex exponential field and by Hrushovski's
(1993) construction of strongly minimal structures (see
\cite{udi}).

Zilber's idea is to have exponential structures which are as existentially closed as 
possible without violating Schanuel's Conjecture. 

\medskip

Zilber proved an important categoricity result for the class of his fields in every uncountable cardinality. He
conjectured that the complex exponential field is the unique model
of cardinality $2^{\aleph_{0}}.$ The ideas contained in
Zilber's axiomatization 
could provide new insights in the analysis of  the complex exponential field.

One of the axioms of Zilber {\textbf{(Strong Exponential Closure)}} is concerned with generic solutions of systems of exponential polynomials, and it is the main obstruction to prove Zilber's conjecture modulo (SC).

In this direction a first result was obtained by Marker for polynomials over $\C$ with only one iteration of exponentation. Using Hadamard Factorization Theorem Marker in  \cite{marker}  proved  the existence of  infinitely many solutions. By restricting to $\Bbb Q^{alg}$ the coefficients of the polynomial and assuming (SC) he showed the existence of infinitely many algebraically independent solutions over $\Bbb Q.$ More recently, Mantova in \cite{mantova} improved Marker's result by eliminating the hypothesis on the coefficients of the polynomial. Schanuel's Conjecture still plays  a crucial role in Mantova's proof. 

In this paper we consider the next natural cases of exponential polynomials with two and three iterations of exponentations, and we obtain an analogous result to that of Marker.\\
Comparing the complex exponential field and Zilber's fields has been one of the main motivation in the following recent papers \cite{dmt2}, \cite{dmt1}, \cite{ayan}, \cite{kmo}.

\section{Masser's result}\label{masser'sresult}

In some hand-written notes (see \cite{masser}) Masser proved the following result. For completeness we give the details of his proof.

\begin{thm}
\label{genMasser1}
Let $P_1(\overline{x}), \ldots , P_n(\overline{x})\in \mathbb C[\overline{x}]$, where $\overline{x}=(x_1, \ldots , x_n),$ and $P_i(\overline x)$ are non zero polynomials in  $\C[\overline{x}]$. Then there exist $z_1,\ldots ,z_n\in \mathbb C$ such that 
\begin{equation}
\label{genMasser}
\left\{ \begin{array}{l}
e^{z_1} = P_1(z_1,\ldots ,z_n) \\
  e^{z_2}  = P_2(z_1,\ldots ,z_n) \\
 \vdots \\
 
 e^{z_n} = P_n(z_1,\ldots ,z_n)
\end{array} \right.
\end{equation} 
\end{thm}

We have to show that the function $F: \mathbb C^n\rightarrow \mathbb C^n$ defined as 

\begin{equation}\label{funzione}
F(x_1, \ldots , x_n)= (e^{x_1}-P_1(x_1, \ldots , x_n), \ldots ,e^{x_n}-P_n(x_1, \ldots , x_n))
\end{equation}

\noindent has a zero in~$\C^n$.
For the  proof we need a result due to Kantorovich (see Theorem 5.3.1 in \cite{kantorovich}) for vector functions in many variables over the reals. Kantorovich's theorem is a refinement of Newton's approximation method for vector functions over the reals, \ie under certain hypothesis the existence of a zero of the function in a neighbourhood of a fixed point is guaranteed. Here we need the following version of Kantorovich's theorem for~$\C$. 

\begin{lemma}\label{Kantorovich}
Let $F:\mathbb C^n\rightarrow \mathbb C^n$ 
be an entire function, and $\overline p_0$ be such that   $J(\overline p_0)$, the Jacobian of $F$ at $\overline p_0$, is non singular. 
Let $\eta= \abs{J(\overline p_0)^{-1}F(\overline p_0)}$ and $U$ the closed ball of center $p_0$ and radius $2\eta$. 
Let $M>0$ be such that $\abs{H(F)}^2 \leq M^2$ (where $H(F)$ denotes the Hessian of $F$). 
If $2M\eta \abs{J(\overline p_0)^{-1}}<1$ then there is a zero of $F$ in~$U$. 
\end{lemma}

\begin{proof} 
Using the canonical transformation ($z=x+iy \mapsto (x,y)$) that identifies $\mathbb C$ with $\mathbb R^2$ we will work with a function $G:\mathbb R^{2n}\rightarrow \mathbb R^{2n}$ which satisfies the hypothesis of Kantorovich's theorem in the case of real variables. Hence  (see Theorem 5.3.1 in \cite{kantorovich})  $G$ has a zero in $\Bbb R^{2n}$ which determines a zero of $F$ in $\Bbb C^n.$ \end{proof}




\begin{lemma}
\label{lower bounds}
Let $P_1(\x), \ldots , P_n(\x)\in \C[\x]$, where $\x = (x_1,\ldots ,x_n)$ 
and $d_1,\ldots ,d_n$ be the total degrees of $P_1(\x), \ldots , P_n(\x)$, respectively. There exists a constant $c>0$ and an infinite set $S\subseteq \mathbb Z^n$  such that 
\[
\abs{P_j(2\pi ik_{1},\ldots , 2\pi ik_{n})} \geq 
c(1+\sum_{l=1}^n \abs{k_{l}})^{d_j}
\]
 for all $\overline{k}=(k_1,\ldots , k_n) \in S$, $j=1, \ldots ,n$.
\end{lemma}

\begin{proof}
We prove the lemma for a single polynomial $P(\x)$ of degree~$d$. 
Let $P(\x)= Q_d(\x)+Q_{d-1}(\x)+\ldots +Q_0(\x)$, 
where each $Q_h(\x)$ is a homogenous polynomial of degree~$h$. 
Fix $\qv = (q_1,\ldots ,q_n)\in \mathbb Z^n$ such that  $Q_d(\qv)\not=0$, and let 
$S= \{ t \qv : t\in \mathbb N\}$. 
We now estimate $P(2\pi it \qv)$ for $t \qv \in S$. 
Clearly,
\begin{multline*}
\abs{Q_d(2\pi i t \qv)+Q_{d-1}(2\pi it \qv)+\ldots +Q_0(2\pi it \qv)}\geq\\ 
\abs{Q_d(2\pi it \qv)}-\abs{Q_{d-1}(2\pi it \qv)+\ldots +Q_0(2\pi it \qv)}.
\end{multline*}
Simple calculations give 
\begin{equation} 
\label{estimateQd}
\abs{ Q_d(2\pi it \qv)}= C_1\abs{t \qv}^d
= C_1(\abs{tq_1}+\ldots +\abs{tq_n})^d
\end{equation}
for some constant~$C_1$.  
By easy estimates we get 
\[
C_1(\abs{tq_1}+\ldots +\abs{tq_n})^d
\geq C_2(\abs{tq_1}+\ldots +\abs{tq_n}+1)^d
\] 
for some constant~$C_2$. 
Notice that all constants in the inequalities depend only on the total degree  and on the coefficients of~$P$. 
\end{proof}
\bigskip

\noindent
{\bf Proof of the Theorem \ref{genMasser1}} Let $S$ as in Lemma \ref{lower bounds}. In order to apply Lemma \ref{Kantorovich} to $F(x_1, \ldots, x_n)$ as in (\ref{funzione}) we choose a point $\kv = (k_1, \dotsc, k_n)$
 in $S$ and we look for a solution of $F$ near $2\pi i \kv$.
 We  first transform the functions defining $F$ by shifting the variables.
Let $~p_1=~P_1(2\pi i \kv), \ldots , p_n=P_n(2\pi i \kv)$. 
Lemma~\ref{lower bounds} guarantees that $p_1, \ldots , p_n$ are different from~$0$. 
Let $a_1, \ldots, a_n$ be the principal values of $\log p_1, \ldots , \log p_n$, respectively. 
If $T=1+|k_1|+\ldots +|k_n|$ then 
 \begin{equation}
 \label{max}
 \max \{ |a_1|, \ldots, |a_n|\} \leq C\log T,
\end{equation} 
for some constant $C$ depending only on the coefficients and the degrees of the polynomials $P_j$'s, and not on the choice of $\kv$ in~$S$. 
We make now a change of variables by shifting each variable $x_j$ by $2\pi i k_j+a_j$, and we solve the new system

 \begin{equation}
\label{Massersystem2}
\left\{ \begin{array}{l}
f_1(x_1,\ldots ,x_n)= e^{x_1} -\frac{P_1(2\pi ik_1+a_1+ x_1, \ldots , 2\pi ik_n+a_n+x_n)}{p_1}=0 \\
  f_2(x_1,\ldots ,x_n)=e^{x_2}  - \frac{P_2(2\pi ik_1+a_1+ x_1, \ldots , 2\pi ik_n+a_n+x_n)}{p_2}=0 \\
 \vdots \\
 f_n(x_1,\ldots ,x_n)= e^{x_n} - \frac{P_n(2\pi ik_1+a_1+ x_1, \ldots , 2\pi ik_n+a_n+x_n)}{p_n}=0
\end{array} \right.
\end{equation} 
We now evaluate the Jacobian of the new system at the point $~\overline p_0=~(0,\ldots ,0)$ (which corresponds to $(2\pi ik_1+a_1, \ldots , 2\pi ik_n+a_n)$ after the shifting). We have

\begin{equation}\label{jacobiano}
J(\overline p_0)= \left( \begin{array}{llcl}
 \partial_{x_1}(f_1) &  \partial_{x_2}(f_1) & \ldots & \partial_{x_n}(f_1)\\
\partial_{x_1}(f_2) & \partial_{x_2}(f_2) & \ldots & \partial_{x_n}(f_2)\\
\vdots & \vdots & \vdots\\
\partial_{x_1}(f_n)& \partial_{x_2}(f_n) & \ldots &
\partial_{x_n}(f_n)
\end{array}\right)
\end{equation}

where $$ \partial_{x_h}(f_h) = 1-  \frac{(\partial_{x_h}P_h)(2\pi ik_1+a_1, \ldots , 2\pi ik_n+a_n)}{p_h}$$ for $h = 1, \ldots, n,$ and $$ \partial_{x_h(f_j)} = -  \frac{(\partial_{x_h}P_j)(2\pi ik_1+a_1, \ldots , 2\pi ik_n+a_n)}{p_j}$$ for all $h \neq j.$

By Lemma \ref{lower bounds} the quotients $  -  \frac{(\partial_{x_h}P_j)(2\pi ik_1+a_1, \ldots , 2\pi ik_n+a_n)}{p_j}$ for all $h, j =1, \ldots, n$ converge to $0$ for large  $T$. Hence, $J(\overline p_0)$ converges to the identity matrix, and so it is not singular. Moreover, also the inverse matrix  $J(\overline p_0)^{-1}$  converges to the identity matrix, so $|J(\overline p_0)^{-1}|$ is bounded by a constant, say~$C_0.$  
We need to evaluate the norm of $F(\overline p_0)$. 
By Lemma~\ref{lower bounds}, equation \eqref{max}, and the mean value theorem (see \cite{lang}) we obtain $|F(\overline p_0)|\leq C_1\frac{\log T}{T}$, for some constant~$C_1.$ 

Hence, 
\begin{equation}\label{stimaM}
\abs{J(\overline p_0)^{-1}F(\overline p_0)} \leq C_2\frac{\log T}{T},
\end{equation}
for some constant~$C_2$. 
Let $\eta = C_2\frac{\log T}{T}$, and let $U$ be the closed ball of center $\overline p_0 = (0, \ldots, 0)$ and radius $2\eta.$  In order to complete the proof, we need to satisfiy the last condition of Lemma \ref{Kantorovich}, \ie 
$2 \abs{J(\overline p_0)^{-1}F(\overline p_0)} M \abs{J(\overline p_0)^{-1}}<1,$ 
for some $M > 0$ bounding the norm of the Hessian of the function $F$ on~$U.$ This inequality follows from (\ref{stimaM}) and the boundness of $|J(\overline p_0)^{-1}|.$ \hfill $\Box$
 \\

\subsection{Generalization to algebraic functions}

Masser in his notes remarked that using the same argument the result can be generalized to algebraic functions. Here we give the proof following Masser's idea.\\

An algebraic function is a complex analytic function (in many variables) defined on some ``cone'' (at infinity) and satisfying a polynomial equation over $\Bbb C$.

More precisely: for us, a cone is an open subset $U \subseteq \C^n$ such that for every $1 \leq t \in \R,$ if $\overline x \in U$ then $t \overline x \in U$.

We denote by $\overline x = (x_1, \dotsc, x_n)$ an $n$-tuple, and by $u$ a single variable.

\begin{defini}
An algebraic function is an analytic function $f:U \to \C$ such that there exists a nonzero polynomial $p(\x, u) \in \C[\x,u]$ with
$p(\x, f(\x)) = 0$ on all $\overline x \in U$. 
If, moreover, the polynomial $p$ is monic in $u$, we say that $f$ is integral algebraic.
\end{defini}

\begin{defini}
Let $f : U \to \C$ (where $U$ is a cone) be an algebraic function. We say that $f$ is homogeneous of degree $r$ if, for every $\x \in U$ and $1 \leq t \in \R$, we have
$f(t \x) = t^r f(\x)$.
\end{defini}

For every algebraic function $f$ there exists a unique $r \in \Q$ (the degree of $f$) and $h: U \to \C$ algebraic and homogeneous of degree $r$, such that
$f(\x) - h(\x) = o(\abs{\x}^r)$.\\

\begin{fact}
Notice that if $f$ is a polynomial, then $f$ is homogeneous in the above sense iff it is homogeneous as a polynomial, and that its degree is equal to the total degree as a polynomial.\\
Moreover, every algebraic function can be expressed as the quotient of two integral functions (after  shrinking the domain, if necessary), and the degree of an integral function is greater or equal to $0$.
\end{fact}

\begin{esempio}
The function $(x_1, x_2) \mapsto \sqrt{x_1} + \sqrt[6]{x_1^3 + x_2^3}$ is integral and homogeneous of degree $1/2$, but is not analytic at infinity.
\end{esempio}

We now state and sketch a proof of a generalization of Theorem \ref{genMasser1} to algebraic functions.

\begin{thm}\label{algebraic}
Let $f_1, \dotsc, f_n :U \to \C$ be nonzero algebraic functions, defined on some cone $U$.
Assume that 
$U \cap (2 \pi i \Z^*)^n$ is Zariski dense in $\C^n$.
Then, the system 
\begin{equation}\label{eq:systemb}
\left\{\begin{aligned}
e^{z_1} &= f_1(\overline z), \\
\dots\\
e^{z_n} &= f_n(\overline z)
\end{aligned}\right.
\end{equation}
has a solution $\overline a \in U$.
\end{thm}

\begin{proof}[Sketch of Proof]
The proof is quite similar to that of Theorem \ref{genMasser1}. We will only show the modifications that are needed in the case of algebraic functions.

First, we make a reduction to the case where all the $f_i'$s are integral (and not only algebraic): it suffices to write $f_i = g_i/h_i$, where $g_i'$s and $h_i'$s are integral
and solve the system in $2 n$ equations and $2 n$ variables
\be\label{eq:system-2n}
\left\{\begin{aligned}
e^{x_1} &= g_1(\overline x - \overline y), \\
\dots\\
e^{x_n} &= g_n(\overline x - \overline y),\\
e^{y_1} &= h_1(\overline x - \overline y), \\
\dots\\
e^{y_n} &= h_n(\overline x - \overline y).
\end{aligned}\right.
\ee
If $(\overline a, \overline b) \in U \times U$ is a solution
of \eqref{eq:system-2n}, then $\overline a - \overline b$ is a solution of \eqref{eq:systemb}.

Let $d_i \in \Q$ be the degree of $f_i.$
Since we have assumed all the $f_i$ are integral, $d_i \geq 0$ for every  $i$.
Write $f_i = h_i + g_i$, with $h_i$ homogeneous of degree $d_i$ and 
$g_i(\x) = o({\abs {\x}}^{d_i}).$
Choose $\overline v \in (2 \pi i \Z^*)^n$ such that, for every $i \leq n$, $c_i = h_i(\overline v) \neq 0$.
Pick $t \in \N$ large enough (we will see later how large), 
and denote $\overline \omega = t \overline v$.
Notice that $f_i(\overline \omega ) = t^{d_i} (c_i + o(1))$ and therefore, for some constant $c >0$ and for  $t$ large enough,
$\abs{f_i(\overline \omega )} \geq c (1 + \abs{\overline \omega })^{d_i}$  (Lemma \ref{lower bounds}).\\
Let $A_i = f_i(\overline \omega )$ and $a_i$ be the principal logarithm of $A_i$.
It is easy to see that $a_i = O(\log t)$.
 As in the proof of Theorem \ref{genMasser1}, let $\overline a = a_1, \ldots, a_n,$ we make the change of variables $\overline z = \overline \omega  + \overline a + \overline x$ and we are reduced to solve the equation $F(\x) = 0$, where
\be\label{eq:masser-algebraic-newvar}
\left \{ \begin{aligned}
F_1(\x) &= e^{x_1} - \frac{f_1(\overline \omega  + \overline a + \x)}{A_1}\\
\dotsc\\
F_n(\x) &= e^{x_n} - \frac{f_n(\overline \omega  + \overline a + \x)}{A_n}
\end{aligned} \right.
\ee
and $F(\x)=(F_1(\x), \ldots , F_n(\x))$.

Finally, for $t$ large enough,  $F$ satisfies the hypothesis of Lemma \ref{Kantorovich} on an open ball of center $\overline 0$ contained in its domain, and we have finished.
\end{proof}

\begin{osserva}
\rm{The above theorem can be generalized to the situation where, instead of being algebraic funtions, $f_1, \dotsc, f_n$ are analytic on $U$ and roots of some nonzero polynomials $P_i(\x, u) \in \mathcal O_n[u]$, where $\mathcal O_n$ is the ring of germs of functions on $\C^n$ analytic in a neigbourhood of infinity.}
\end{osserva}
Given polynomials $p_1(\overline x, u), \dotsc, p_n(\overline x,u)$ of degree at least $1$ in $u$, there exists a nonempty cone $U$ and algebraic functions $$f_1, \dotsc, f_n: U \to \C,$$ such that $p_i(\overline x, f_i(\overline x)) =  0$ on all $U$.
Moreover, since $(2 \pi i \Z^*)^n$ is Zariski dense in $\C^n$, we can also find $U$ as above such that $(2 \pi i \Z^*)^n \cap U$ is also Zariski dense.
Thus, in order to find a solution of a system $$p_1(\overline x, e^{x_1}) = 0, \dotsc, p_n(\overline x, e^{x_n}) = 0,$$ it suffices to find $\overline a \in U$ such that $e^{a_1} = f_1(\overline a), \dotsc, e^{a_n} = f_n(\overline a)$, and we can apply the above theorem to find such $\overline a.$

Let $G_n(\C) = \C^n \times (\C^{*})^n$ be the algebraic group. We have the following result.
\begin{corollario}\label{lem:system-poly}
Let $p_1, \dotsc, p_n \in \Bbb C [\overline x,u]$ be nonzero irreducible polynomials of degree at least $1$ in $u$, and not of the form a constant times $u$.
Let $V \subseteq G_n(\C)$ be an irreducible component of the set
\[
\set{(\x, \overline y) \in G_n(\C): \bigwedge_{i=1}^n p_i(\x, y_i) = 0}.
\]
Assume that $\pi(V)$ is Zariski dense in $\C^n$ (where $\pi: G_n(\C) \to \C^n$ is the projection onto the first $n$ coordinates).
Then, the set $\set{\overline a \in \C^n: (\overline a, e^{\overline a}) \in  V}$ is Zariski dense 
in $\C^n$.
\end{corollario}

\begin{proof}
Let $W \subset \C^n$ be a Zariski open subset.
Let $U$ be a cone and $f_1, \dotsc, f_n : U \to \C$ be algebraic functions, such that
$U \cap (2 \pi i \Z^*)^n$ is Zariski dense in $\C^n$, $U$ is contained in~$W$,
and  $p_i(\x, f_i(\x)) = 0$ for every $\x \in U$.
Choose $\overline a$ solving system \eqref{eq:systemb} (the conditions on the polynomials $p_i$'s ensure that the $f_i$'s exist, and are nonzero).
Then $(\overline a, e^{\overline a}) \in V$ and $\overline a \in W$.
\end{proof}

We can generalize the above lemma.
\begin{lemma}
Let $W \subseteq G_n(\C)$ be an irreducible algebraic variety such that $\pi(W)$ is Zariski dense in $\C^n$ 
(where $\pi: G_n(\C) \to \C^n$ is the projection onto the first $n$ coordinates)%
\footnote{This is a non-trivial condition. A major problem is to replace this condition with much weaker ones while still retaining the conclusion of the Lemma.} .Then, the set $\set{\overline a \in \C^n: (\overline a, e^{\overline a}) \in W}$ is Zariski dense in $\C^n$. 
\end{lemma}

\begin{proof}
There exist polynomials $p_1, \dotsc, p_n \in \C[\overline x,u]$ and $V$  irreducible component of
$\set{(\overline x, \overline y) \in G_n(\C): \bigwedge_{i=1}^n p_i(\overline x, y_i) = 0}$
 satisfying the hypothesis of Corollary \ref{lem:system-poly}, and moreover with $V \cap W$ Zariski dense in $V$. 
Thus, using Corollary \ref{lem:system-poly} we complete the proof.
\end{proof}

\section{Zeros of exponential polynomials over $\C$}
\label{casocomplesso}

Let $(R, E)$ be an exponential ring. The ring of exponential polynomials over $(R, E)$ in $z_1, \dots, z_n$  variables is defined by recursion, and it is denoted by $R[z_1, \dots, z_n]^E$ (for details see \cite{van}).\\
Henson and Rubel in \cite{hen} gave a characterization of those exponential polynomials over $\mathbb C$ with no roots.
Their proof is based on Nevanlinna theory.

\begin{thm}
\label{henson-rubel}
\cite{hen}
 Let $F(z_1, \ldots, z_n) \in \Bbb C[z_1, \ldots, z_n]^E$
$$F(z_1, \ldots, z_n) \mbox{ has no roots in } \Bbb C \mbox{ iff } F(z_1, \ldots, z_n) =
e^{G(z_1, \ldots, z_n)},$$ where $G(z_1, \ldots, z_n) \in \Bbb
C[z_1, \ldots, z_n]^E.$
\end{thm}

Katzberg in \cite{katz} using Nevanlinna theory and considering exponential polynomials in one variable proved the following result:

\begin{thm} \cite{katz} \label{katzberg} A non constant exponential polynomial $F(z) \in \Bbb C[z]^E$ has always infinitely many zeros unless it is of the form
\[ 
F(z) = (z-\alpha_1)^{n_1} \cdot \ldots \cdot (z-\alpha_n)^{n_n} e^{G(z)},
\] 
where $\alpha_1, \ldots, \alpha_n \in \Bbb C, n_1, \ldots, n_n \in \Bbb N,$ and $G(z) \in \Bbb C[z]^E.$
\end{thm}

In \cite{dmt},  using purely algebraic methods, the two previous theorems have been proved for exponential polynomials over a Zilber's field. 

\smallskip

\par We now investigate some special cases of the axiom of Strong Exponential Closure, with the aim of proving that they are 
true in $(\C, +, \cdot, 0, 1,e^z).$ Marker in \cite{marker} proved the first result in this direction
for polynomials in $z, e^z$ over $\Bbb Q^{alg}.$

More precisely he showed:

\begin{thm}\cite{marker}
If $p(x, y) \in \Bbb C[x, y]$ is irreducible and depends on $x$
and $y$ then $f(z) = p(z, e^z)$ has infinitely many zeros.

\par Moreover, assuming (SC), if $p(x, y) \in \Bbb Q^{alg}[x,
y],$ and  $p(z, e^{z}) = p(w, e^{w}) = 0$ with $z, w \not= 0,$ and $z \not = \pm w$  then $z, w$ are algebraically independent over~$\Q.$
\end{thm}

The first part of the theorem follows from Hadamard Factorization
Theorem (see \cite{lang}), which he can apply to $f$ since 
$f$ has order one.



Recently, in \cite{mantova} Mantova proved that assuming Schanuel's Conjecture any polynomial $p(z, e^z)$ with coefficients in $\Bbb C$ has solutions of maximal transcendence degree in the following sense: for each  finitely generated subfield $K$ of $\Bbb C$ if $p(x, y) \in K[x, y]$ then there is an $a \in \Bbb C$ such that $$p(a, e^a) = 0, \mbox{ and } t.d._K(a, e^a) =1.$$

The next natural case to consider is that of a polynomial  $p(z, e^{e^z})$ with two
iterations of exponentiation. Hadamard Factorization
Theorem cannot be applied anymore since the function $f(z) =
p(z, e^{e^z})$ has infinite order. 

\begin{thm}\label{esistenza}
Let $f(z) = p(z, e^z, e^{e^z}, \ldots, e^{e^{e^{\ldots^{e^{z}}}}}),$ where $p(x, y_1 \ldots, y_k)$ is an irreducible polynomial over $\mathbb C[x, y_1 \ldots, y_k]$. The function $f$ has
infinitely many solutions in $\Bbb C$ unless $p(x, y_1 \ldots, y_k) = g(x) \cdot y_1^{n_{i_1}}\cdot \ldots \cdot y_k^{n_{i_k}},$ where $g(x) \in \Bbb C[x].$

\end{thm}

\begin{proof}
It is an immediate consequence of Theorem \ref{katzberg}. An alternative proof is obtained easily by applying Theorem \ref{algebraic}.
\end{proof}

In the sequel we will always assume that the polynomial $p$ is not of the form $p(x, y_1 \ldots, y_k) = c \cdot y_1^{n_{i_1}}\cdot \ldots \cdot y_k^{n_{i_k}},$ where $c \in \Bbb C,$ and we assume $p(x, y_1 \ldots, y_k)$ is also an irreducible polynomial over $\mathbb C[x, y_1 \ldots, y_k].$

\section{Generic solutions}

Let $e_0(z) = z$, and for every $k \in \N$, define $e_{k+1}(z) = e^{e_k(z)}$.
Fix $1 \leq k \in \N$, let $\overline{x} = (x_0, \dotsc, x_k)$ and $p(\overline x) \in \Q^{alg}[\overline{x}].$ We assume the polynomial $p$ irreducible, and depending on $x_0$ and the last variable.
An element $a \in \C$ is a generic solution of
\be\label{eq:n-exp}
f(z) = p(z, e_1(z), \dotsc, e_k(z)) = 0
\ee
if $\td_{\Bbb Q}(a, e_1(a), \dotsc, e_k(a)) = k$.\\
In this section we investigate the existence of a generic solution $a$ of (\ref{eq:n-exp}).


We always assume Schanuel's Conjecture.  Our proof is crucially based on Masser's result (see Section~\ref{masser'sresult}).\\





\subsection{The function $f(z) = p(z, e^{e^{z}})$}

\par The first case we consider is  when the exponential polynomial $f(z)$ has two
iterations of exponentiation. In particular, we want to answer the following questions:

\begin{enumerate}
\item Let $p(x, y) \in \Bbb C[x, y].$ Is there some $w \in \C$ so that $(w, e^{e^w})$ is a generic point of the curve
 $p(x, y) = 0?$
\item What is the transcendence degree of the set of
solutions of $f(z)$?

\end{enumerate}

\noindent For this purpose we consider the corresponding system in four variables $(z_1,z_2,w_1,w_2)$:



\begin{equation}\label{vari1}
V =\left\{ \begin{array}{ll}
p(z_1, w_2) = 0\\
w_1 = z_2
\end{array} \right.
\end{equation}
thought of as an algebraic set $V$ in $G_2(\mathbb C) = \Bbb C^2 \times (\Bbb C^{*})^2$.

\begin{thm}\label{dueiterazioni}

(SC) If $p(x, y) \in \mathbb Q^{alg}[x,y]$ then 
 the variety defined by $V$ intersects the graph of exponentation in a generic point $(w, e^w, e^w, e^{e^w})$, (i.e. $t.d._{\Bbb Q}(w, e^w, e^w, e^{e^w}) = \dim V=2$). 

\end{thm}

\noindent {\bf Proof:} 
By Theorem \ref{esistenza}  the function $f(z) = p(z, e^{e^{z}})$ has a solution $w$ in $\Bbb C.$ If $w=0$ then
$e^{e^0}=e$, and from $p(0,e)=0$ it follows that $p(x,y)$ is a polynomial in the variable $y$. Then $e$ is algebraic over $\mathbb Q$ which is clearly a contradiction. 

\par So, without loss of generality, $w\not=0$.  

We now assume (SC). The point 
 $(w, e^w, e^w, e^{e^w})$ belongs to the variety $V$ associated to system (\ref{vari1})
which has dimension $2$. We distinguish two cases.\\
\textbf{Case 1.}    Assume that $w$ and $e^w$ are linearly independent. By
Schanuel's Conjecture we have:
$$t.d._{\Bbb Q}(w, e^w, e^w, e^{e^w}) \geq 2.$$
\par Indeed, the transcendence degree is exactly 2 since $w$ and $e^{e^w}$ are algebraically dependent.
Hence, $(w, e^w, e^w, e^{e^w})\in V$ and $t.d._{\Bbb Q}(w, e^w, e^w, e^{e^w}) = 2,$ which is  the dimension of $V$,
and so the point $(w, e^w, e^w, e^{e^w})$ is  generic for $V$.\\
\textbf{Case 2.} Suppose that $w,$ $e^w$ are linearly
dependent over $\mathbb Q$. This means that
\begin{equation}
\label{lindip}
ne^w = mw
\end{equation}
for some $m, n \in \Bbb Z$ and $(m, n) = 1.$ Since $w\not = 0$ so necessarily $n \not = 0.$  Moreover, $w$ is
transcendental over $\Bbb Q,$ otherwise we have a contradiction with a Lindemann Weierstrass Theorem. Applying exponentation to relation (\ref{lindip}) it follows
$$e^{ne^{w}}=e^{mw},$$ i.e. $$(e^{e^{w}})^n=(e^w)^m=(\frac{m}{n}w)^m=(\frac{m}{n})^mw^m.$$
We now distinguish the cases, when both $n,m$ are positive, and the case when $n>0$ and $m<0$. We have that  $(w,e^{e^{w}})$ is a root of either $q(x,y)=x^n-sy^m$ or $q(x,y)=x^ny^m-r,$ where  $s,r \in \mathbb Q$.
In both cases the polynomial $q(x,y) $ is irreducible, this is due to the fact that $(n,m)=1$ (see Corollary of Lemma 2C in \cite{Schmidt}).

\par
Let $V(p)$ and $V(q)$ be the varieties associated  to $p$ and $q$, respectively. Clearly, $\dim V(p)=\dim V(q)=1$. There is a point $(w,e^{e^{w}})$ which belongs to both varieties. Moreover, we know that 
every solution $(w,e^{e^{w}})$ of the polynomial $p$ is such that $w$ is transcendental, and this means that the point is generic for the variety $V(q).$ This implies that $V(q) \subseteq V(p),$ hence $p$ divides $q.$ By the irreducibility of both polynomials we have that 
 $p$ and $q$ differ by a non-zero constant.  Without loss of generality we can assume 
\begin{equation}
\label{s}
p(x,y)= q(x,y)= x^n-sy^m
\end{equation}
(the case of $p(x,y)= q(x,y)= x^ny^m-r$ is treated in a similar way). Notice that for any solution $(w,e^{e^{w}})$  of $p(x,y)=0$ the linear dependence between $w$ and $e^w$ is uniquely determined by the degrees of $x$ and $y$ in $p$, hence $s$ in (\ref{s}) is uniquely determined. 
We will show that it is always possible to find a solution $(w, e^w, e^w, e^{e^w})$ of system (\ref{vari1})  with $ w, e^w$ linearly independent. 
Indeed, we consider the following system 

\begin{equation}
\label{Massersystem3}
\left\{ \begin{array}{l}
p(z, e^{e^z}) = 0 \\
   z  \neq se^z
\end{array} \right.
\end{equation} 

\noindent that we can reduce to the following:

\begin{equation}
\label{Massersystem}
\left\{ \begin{array}{l}
e^z = A(z,t,u) \\
  e^u = B(z,t,u) \\
 e^t = C(z,t,u)
\end{array} \right.
\end{equation} 
where $A(z,t,u) = \frac{t}{m}$, $B(z,t,u)= \frac{t}{m} -sz$, and $C(z,t,u)= \frac{z^n}{s}$. By Theorem \ref{genMasser1} there exists a solution of system ($\ref{Massersystem}$) which is generic since the second equation in (\ref{Massersystem}) guarantees that there is no linear dependence between a solution $z$ and its exponential $e^z.$


\subsection{The function $f(z) = p(z, e^z, e^{e^z})$}

Now we examine the more general case $f(z) = p(z, e^z, e^{e^z}).$
\par For this purpose we consider the corresponding system in four variables $(z_1,z_2,w_1,w_2)$:

\begin{equation}\label{vari}
V =\left\{ \begin{array}{ll}
p(z_1, z_2, w_2) = 0\\
w_1 = z_2
\end{array} \right.
\end{equation}
thought of as an algebraic set $V$ in $G_2(\mathbb C)$.


\begin{thm}\label{lem:generic-exp2}
 (SC) If $p(x, y, z) \in \Bbb Q^{alg}[x, y, z]$ and depends on $x$ and $z$, then the variety $V$ defined in (\ref{vari}) intersects the graph of exponentation in a generic point.
\end{thm}
\begin{proof}
By Theorem \ref{esistenza}, there exists $a \in \C$ such that $f(a) = 0$
and $a \neq 0$. Moreover, as in the previous case, by Lindemann-Weierstrass Theorem, $a$ is transcendental over $\Q.$
Also in this case  $\dim V = 2.$ We will show that $\td_{\Bbb Q}(a, e^a, e^a, e^{e^{a}}) = 2$, then $(a, e^a, e^a, e^{e^{a}})$ is a generic point of $V$.
If $\td_{\Bbb Q}(a, e^a, e^a, e^{e^{a}}) = 1$ then by Schanuel's Conjecture, there exists $r \in \Q$ such that

\begin{equation}
e^a = ra.
\end{equation}
 We call $r \in \Q$ ``bad'' if there exists $a\in \C$ solution of
(\ref{vari}), such that $e^a = r a$.

We claim that there exist only finitely many bad $r \in \Q$.
Let $r \in \Q$ be bad.
Assume $r = n/m$, with $0 \neq n \in \Z$, $0 < m \in \N$, and $(n, m) = 1$.
We have

\begin{equation}\label{rela}
me^a = n a
\end{equation}
for some $a \in \Bbb C,$ and therefore
\[
(e^{e^{a}})^m = ( e^a)^n = (r a)^n.
\]

For every ``bad'' rational $r$,  the polynomial $p(x, r x, z)$ becomes into two variables $x, z,$ and we denote it by $p_{r}(x, z).$
Notice that $p(x, r x, z)$ may have become reducible. 

{\textbf{Case 1.}} Assume $n > 0.$ Let $q(x, z) = z^m - (r^n) x^n$ and $V(p_r)$ and $V(q)$ be the varieties associated respectively  to $p_r$ and $q.$ We note that the polynomial $q(x, z)$ is irreducible (see Corollary of Lemma 2C \cite{Schmidt}). The point  $(a, e^{e^a})$ belongs to both varieties, and it is generic for the variety $V(q),$ since $a$ is transcendental. This implies that $V(q) \subseteq V(p_r), $ hence the polynomial $p_r$ divides $q.$ In this case we cannot infer that $q$ and $p_r$ differ by a constant since $p_r$ may be reducible. Thus, either $p_{r} \equiv 0$, or $\deg(p_{r}) \geq  \max(n, m).$
In the first case, since $p$ is nonzero, there exist only finitely many 
$r \in \Q$ such that $p(x, rx, z) \equiv 0$. In the second case, since $\deg(p_r) \leq \deg p$, we have that $ \max(n, m) \leq \deg p$.
Thus in both cases there are only
finitely many bad~$r$'s.

{\textbf{Case 2.}} Assume $n < 0.$  Let $q(x, z) = z^mx^{-n} - r^{n}.$
Since $q$ is an irreducible polynomial, we can argue as in
the case $n > 0$ and conclude that there are only finitely many possible bad~$r$'s.

Let $\set{r_1, \dots, r_k}$ be the set of bad
rational numbers.
Now we use Masser's idea. Consider the system 

\be\label{eq:2-exp-generic}
\left\{\begin{aligned}
e^z  &= f_1(z, t, u_1, \ldots, u_k)\\
e^t &=f_2(z, t, u_1, \ldots, u_k) \\
e^{u_1} &=f_3(z, t, u_1, \ldots, u_k)\\
\dots\\
e^{u_k} &=f_{k+2}(z, t, u_1, \ldots, u_k) .
\end{aligned}\right.
\ee

where $f_1 = t, f_3 = t-r_1z, \ldots f_{k+2} =   t-r_kz,$   and $f_2$ is the algebraic function which solves $z$ in  the original polynomial $p(x, y, z) = 0$.  
By Theorem \ref{algebraic}, \eqref{eq:2-exp-generic} has a solution $(b, e^b, e^b, e^{e^{b}})$ which is a generic solution for \eqref{vari}, since the last $k$ equations guarantee that there is no linear dependence between $b$ and $e^b$ .
\end{proof}


\subsection{General case $f(z) = p(z, e^z, e^{e^z}, \ldots, e^{e^{e^{\ldots^{e^{z}}}}})$}


For the general case, assuming (SC), we have only partial results (see Proposition \ref{gradotras2}).

\begin{lemma}\label{gradotras}
(SC) Let $n \geq 2$ and $f_1, \dotsc, f_n$ be nonzero algebraic functions  
over $\Q(x_1, \dotsc, x_n),$ defined over some cone $U.$ Assume $U \cap (2 \pi i \Bbb Z^{*})^n$ is Zariski dense in $\Bbb C^n,$ and
$\deg(f_1) \neq 0.$ The system
\be
\left\{
\begin{aligned}\label{sistem}
e^{x_1} &= f_1(\overline{x})\\
\dotsc\\
e^{x_n} &= f_n(\overline{x})
\end{aligned}
\right.
\ee
has a solution $\overline{a} \in \C^n$ satisfying $\td_{\Bbb Q}(\overline{a}) \geq 2$.
\end{lemma}

\begin{proof}
For every $i \leq n$, let $d_i = \deg( f_i)$, 
then,
\[
f_i = h_i + \epsilon_i  
\]
for a unique homogeneous algebraic function $h_i$ of degree~$d_i$ and $\deg(\epsilon_i)<d_i.$
Consider the system 

\be\left\{
\begin{aligned}\label{eq:masser-only-generic}
e^{x_1} &= f_1(\overline{x})\\
\dotsc\\
e^{x_n} &= f_n(\overline{x})\\
h_1(\overline{x}) &\neq 0 \\
\dotsc\\
h_n(\overline{x}) & \neq 0\\
d_1 x_2 - d_2 x_1 & \neq 0.
\end{aligned}
\right.
\ee

\noindent which can be easily reduced to a Masser's system. Let $\overline{a}$ be a solution of system (\ref{sistem}).
 We now prove that  $\td_{\Bbb Q}(\overline{a}) \geq 2$. Assume, by a contradiction, that $t.d._{\Bbb Q}(\overline{a}) \leq 1$. By Lindemann-Weierstrass
Theorem, necessarily we have  $t.d._{\Bbb Q}(\overline{a}) = 1$, and by Schanuel's Conjecture, $\overline{a}$ has $\Q$-linear
dimension~$1$.
Thus, there exist $\overline{m}_1, \dotsc, \overline{m}_{n-1} \in \Z^n$ which are $\Q$-linearly
independent, and such that
\[
\overline{m}_j \cdot \overline{a} = 0, \quad j=1 \dotsc n-1.
\]
We have
\[
\hat f(\overline{a})^{\overline{m}_j}  = f_1(\overline a)^{m_{j_{1}}} \cdot \ldots \cdot f_n(\overline a)^{m_{j_{n}}} = e^{\overline{m}_j\cdot \overline{a}}= 1
, \quad j=1 \ldots n-1.
\]
Let
\[
L = \set{\overline{z} \in \C^n: \bigwedge_{j=1}^{n-1} \overline{m}_j \cdot \overline{z} = 0}.
\]
Clearly, $L$ is a $\C$-linear space of dimension $1$, and 
$\overline{a} \in L$. Thus, $L$ is the $\C$-linear span of~$\overline{a}$.
Moreover, since $t.d._{\Bbb Q}(\overline{a}) = 1$,  
for every $t \in \C$ such that
$f_i(t \overline a) \neq 1$, for every $i \leq n$, and we have
\[
\hat f(t \overline a)^{\overline{m}_j} = 1, \quad j = 1 \dotsc n - 1.
\]
For $t \in \R,$ $t >> 1$,
since $h_i(\overline a) \neq 0$ for every~$i$, 
we obtain
\[
\overline{m}_j\cdot \overline{d} = 0, \quad j=1 \dotsc n-1
\]
where $\overline d = (d_1, \dotsc, d_n)$.
Thus, $\overline d \in L.$ Since $L$ has $\C$-linear dimension~$1$, 
we have $\overline a = \lambda \overline d$ for some $\lambda \in \C$,
contradicting our choice $d_1 a_2 \neq d_2 a_1$.
\end{proof}

Clearly Lemma~\ref{gradotras} implies the following: 

\begin{corollario}\label{gradotras1}
(SC) Let $n \geq 2$. Let $p_1(\overline x), \dotsc, p_n(\overline x) \in \Bbb Q^{alg}[\overline x]$ be nonconstant polynomials in $\overline x = (x_1, \dotsc, x_n)$. 
Then, the system
\be
\left\{
\begin{aligned}\label{eq:masser-generic}
e^{x_1} &= p_1(\overline x)\\
\dotsc\\
e^{x_n} &= p_n(\overline x)
\end{aligned}
\right.
\ee
has a solution $\overline a$ such that $t.d._{\Bbb Q}(\overline a) \geq 2$.
In particular, if $n = 2$, then \eqref{eq:masser-generic} has a generic solution.
\end{corollario}

\begin{osserva}
The hypothesis in Lemma \ref{gradotras} and Corollary \ref{gradotras1} are minimal in order to ensure that $t.d._{\Bbb Q}(\overline a) \not = 0,1.$  \end{osserva}

Adding some extra hypothesis we strength Corollary \ref{gradotras1} as follows.


\begin{lemma}
(SC) Let $p_1(\overline x), \dotsc, p_n(\overline x) \in \Q^{alg}[x_1, \dotsc, x_n]$.
Let $c_i = p_i(\overline 0)$.
Assume that the $c_i$ are nonzero and multiplicatively independent
(\ie, for every $\overline 0 \neq \overline m \in \Z^n$, \mbox{${ \hat c}^{\,\overline m} \neq 1)$}.
Then, all solutions of  the system
\be\label{eq:mult}
\left\{\begin{aligned}
e^{x_1} &= p_1(\overline x)\\
\dots\\
e^{x_n} &= p_n(\overline x)
\end{aligned}\right.
\ee
are generic.
\end{lemma}
\begin{proof}
Let $\overline a \in \C^n$ be a solution of \eqref{eq:mult} and let $k = n - t.d._{\Q}(\overline a).$
Assume, by contradiction, that $k > 0$.
By (SC), there exist $\overline m_1, \dotsc, \overline m_k \in \Z^n$ linearly independent, such that
$\overline m_1 \cdot \overline a = \dots = \overline m_k  \cdot \overline a = 0$.
Thus, $(e^{\overline a})^{\overline m_1} = \dots = (e^{\overline a})^{\overline m_k} = 1$, and therefore
$\hat p(\overline a)^{\overline m_1} = \dots = \hat p(\overline a)^{\overline m_k} = 1,$ where $\hat p(\overline a)^{\overline m_j} = p_1(\overline a)^{m_{j_{1}}} \cdot \ldots \cdot p_n(\overline a)^{m_{j_{n}}},$ for $j = 1, \ldots, k.$ 
Let $L = (\overline m_1)^\perp \cap \dots \cap (\overline m_k)^\perp$.
Thus, $L$ is a linear space of dimension $n-k$ defined over $\Q$.
Since $\overline a \in L$ and $\td_{\Q}(\overline a) = \dim(L)$, we have that $a$ is a generic point of $L$. 
Thus,
$\hat p(\overline x)^{\overline m_j} = 1$ on all $L$, for $j = 1, \dotsc, k$.
In particular, $ \hat c^{\overline m_j} = \hat p(\overline 0)^{\overline m_j} = 1$, contradicting the assumption that the $c_i$'s are multiplicatively independent.
\end{proof}

Now we are able to prove the following result.

\begin{prop}\label{gradotras2}
 (SC) There is a solution $a \in \Bbb C$ of  (\ref{eq:n-exp}) such that $$t.d._{\Bbb Q}(a, e_1(a), \ldots, e_k(a))  \not= 0, 1, k-1.$$
\end{prop}

\begin{proof} 
As in the previous cases, $\td (a, e_1(a), \dotsc, e_k(a))\not=0$ because of Lindemann-Weierstrass Theorem.\\
In order to prove that $\td _{\Bbb Q}(a, e_1(a), \dotsc, e_k(a))\not=1$ it is enough to apply Lemma \ref{gradotras}.\\
Assume now that $\td(a, e_1(a), \dotsc, e_k(a)) =k-1$. By (SC) there exists a $k$-tuple $\overline{0} \neq (m_0, \dotsc, m_{k-1}) \in \Z^{k}$ (and without loss of generality we can assume $m_{k-1}\not=0$) such that 

$$m_{k-1}e_{k-1}(a) = \sum_{i = 0}^{k-2} m_i e_i(a) = \tilde{m} \cdot \at$$ 
where $\tilde{m}=(m_0, \dotsc, m_{k-2}) \in \Z^{k-1}$ and  $\at = (a, e_1(a), \dotsc, e_{k-2}(a))$.  Then the following relations hold

\begin{enumerate}
\item 
$e_{k-1}(a) = \sum_{i = 0}^{k-2} \frac{m_i }{m_{k-1}}e_i(a) = \frac{\tilde{m}}{m_{k-1}} \cdot \at$

\item
$ e_{k}(a)^{m_{k-1}} =  e_1(a)^{m_0} e_2(a)^{m_1} \ldots   e_{k-1}(a)^{m_{k-2}}$.
\end{enumerate}

Let  $\tilde{r}=(r_0, \ldots , r_{k-2})=( \frac{m_0}{m_{k-1}}, \ldots , \frac{m_{k-2}}{m_{k-1}})$, and $\xt = (x_0, \dotsc, x_{k-2})$. 
Let $I_1$, $I_2$ be the partition of $\{  1, \ldots , k-2\}$ induced by $\tilde{m} $, i.e. $I_1$ is the set of those indices $i$ corresponding to negative $m_i$'s, and $I_2$ is the set of those indices $j$ corresponding to positive $m_j$'s. 
Define the following two polynomials

\

$g_{\tilde{r}}(\tilde{x},z)= p(\tilde{x},\tilde{r}\cdot  \tilde{x}, z)$, 

\

$s_{\tilde{m}}(\tilde{x},z)=z^{m_{k-1}} \prod_{i\in I_1}x_i^{-m_i}- \prod_{j\in I_2}x_j^{m_j}$

\

For convenience we consider the polynomial $s_{\tilde{m}}(\tilde{x},z)$ also in the variable $x_0$ even if this variable does not appear. We notice that the polynomial $g_{\tilde{r}}$ may be reducible, while  $s_{\tilde{m}}$ is irreducible (see \cite{Schmidt}).

\
We call a tuple $(\frac{m_0}{m}, \ldots ,\frac{m_{k-2}}{m})\in \mathbb Q$ {\it bad} if there exists $a\in \mathbb C$ solution of (\ref{eq:n-exp}) such  that 
$$e_{k-1}(a) = \sum_{i = 0}^{k-2}\frac{m_i}{m} e_i(a).$$ 

\

Notice that
$(\at, e_{k}(a))$ is a  solution of  both $g_{\tilde{r}}(\tilde{x},z)=0$ and $s_{\tilde{m}}(\tilde{x},z)=0$, and it is generic for $g_{\tilde{r}}(\tilde{x},z)=0$. Hence, $s_{\tilde{m}}$ divides $g_{\tilde{r}},$ and as in Theorem \ref{lem:generic-exp2} there are only finitely many bad tuples of such rationals.

Arguing as in Theorem~\ref{lem:generic-exp2} we consider a new system as (\ref{eq:2-exp-generic}) which has a solution that is a generic solution for (\ref{eq:n-exp}).
\end{proof}

\smallskip

\begin{corollario}
 (SC) Let $f(z) = p(z, e^z, e^{e^z}, e^{e^{e^{z}}})$, where $p(x,y,z,w)\in \mathbb Q^{alg}[x,y,z,w]$ dependes on the last variable. Then there is $a \in \Bbb C$ which is generic solution for $f(z) = 0$.
\end{corollario}



{\bf Acknowledgements.} This research is part of  the project FIRB 2010, {\it Nuovi sviluppi nella Teoria dei Modelli dell'esponenziazione}. The authors thank   
A.~Macintyre, V.~Mantova and D.~Masser for many helpful discussions.

\end{document}